\documentclass[12pt,reqno]{amsart}
\usepackage{amsmath}
\usepackage{amssymb}  
\usepackage{latexsym} 
\usepackage{comment}
\usepackage{url}
\usepackage{fullpage,url,amssymb,amsmath,amsthm,amsfonts,mathrsfs}
\usepackage[usenames,dvipsnames]{color}
\usepackage[pagebackref = true, colorlinks = true, linkcolor = blue, citecolor = Green]{hyperref}
\usepackage[alphabetic,lite]{amsrefs}
\usepackage{enumitem}
\usepackage{amscd}   
\usepackage[all, cmtip]{xy} 
\usepackage{xfrac}
\usepackage[T1]{fontenc}
\usepackage{algpseudocode,algorithm,algorithmicx}

\DeclareFontEncoding{OT2}{}{} 

\usepackage[all]{xy}
\usepackage{fullpage}

\usepackage{color} 

\newcommand{\todo}[1]{{\color{blue} \sf TO DO}}

\def\act#1#2%
  {\mathop{}%
   \mathopen{\vphantom{#2}}^{#1}%
   \kern-3\scriptspace%
   #2}

\newcommand{\F}{{\mathbb F}}

\newcommand{\PP}{{\mathbb P}}

\newcommand{\al}{{\alpha}}

\newcommand{\calO}{{\mathcal O}}

\newcommand{\fp}{{\mathfrak p}}

\newcommand{\fm}{{\mathfrak m}}

\newtheorem{Theorem}{Theorem}[section]
\newtheorem{Lemma}[Theorem]{Lemma}
\newtheorem{Proposition}[Theorem]{Proposition}

\newtheorem{Remark}[Theorem]{Remark}

\theoremstyle{definition}

\numberwithin{equation}{section}

\begin{document}
\title{Factoring polynomials over function fields}

\author{Jos\'e Felipe Voloch}
\address{School of Mathematics and Statistics, University of Canterbury, Private Bag 4800, Christchurch 8140, New Zealand}
\email{felipe.voloch@canterbury.ac.nz}
\urladdr{http://www.math.canterbury.ac.nz/\~{}f.voloch}

\begin{abstract}
If $K/k$ is a function field in one variable of positive characteristic, 
we describe a general algorithm
to factor one-variable polynomials with coefficients in $K$. The algorithm
is flexible enough to find factors subject to additional restrictions, e.g., 
to find all roots that belong to a given finite dimensional $k$-subspace of $K$,
more efficiently. For bounded characteristic, it runs in polynomial time, relative to factorizations
over the constant field $k$ and also provides a deterministic polynomial time irreducibility test. 
We also discuss applications to places of reducible reduction, when $k$ is a global field, and to list decoding of Reed-Solomon codes. 
\end{abstract}

\maketitle

\section{Introduction}
\label{intro}

Let $K/k$ be a function field in one variable, that is, a finitely generated extension of transcendence degree one with
$k$ algebraically closed in $K$. Let $G(T)$ be a polynomial in one variable over $K$. The algorithmic problem
of finding the irreducible factors of $G(T)$ in $K[T]$ and, in particular its roots in $K$, is a much-studied problem with
many applications, see e.g. \cites{MR790658,MR2168610,MR2537701} and the references therein. 
A noteworthy special case is the case where $K = k(x)$,
the rational function field, and the coefficients of $G(T)$ are in $k[x]$. This case corresponds to factoring polynomials in
two variables with coefficients in $k$. Some recent papers representing the state of the art of this 
special case are \cites{MR2600710,MR3709331}.

Many of the applications of the above problem actually require the solution of a more restricted problem. For instance, given $G(T)$ and a finite dimensional $k$-subspace $V$ of $K$, find the roots of $G(T)$ in $V$.
An example of an application where this restricted problem suffices is the Guruswami--Sudan list-decoding algorithm. See \cite{Gao} for a comprehensive discussion and \cites{MR3128698,MR3703706} for other approaches to this problem. We discuss an instance of this in subsection \ref{subsec:list} below.

Throughout this paper $k$ has characteristic $p>0$ and we make the assumption that the polynomial $G(T)$ to be factored is separable. The reduction to this case 
is a standard first step in all algorithms and is presented in the above cited papers. We also assume $G(0) \ne 0$. 
This paper
describes an algorithm that solves the general problem of factoring such a
$G(T)$. Additionally, the algorithm has 
improved performance when applied to the more restricted problems described above.
Indeed, we will describe an algorithm that finds a factor of $G(T)$ of prescribed
degree whose coefficients are on prescribed finite dimensional $k$-subspaces of $K$.  
We prove that, assuming the characteristic is bounded, our 
algorithm provides a deterministic polynomial time absolute irreducibility test
and, up to factorization of polynomials in $k[T]$, the algorithm runs in deterministic polynomial time. See Theorem \ref{thm:main},
Remark \ref{rem:constant} and Section \ref{sec:place} below for precise statements and discussion.
We also discuss some other applications of the ideas in this paper in Section \ref{sec:appl}.
The approach is novel and relies on casting the problem of finding a factor of $G(T)$ as finding a $k$-linear dependence relation among some elements of a quotient ring of $K[T]/(G(T))$ and applying a linear independence criterion over $k$
involving Wronskian matrices. 

Our algorithm, like most standard algorithms, needs at the beginning, a place of the function field $K$ with
a few additional properties. In our presentation of the main part of the
algorithm, in Section \ref{sec:algo}, we just assume its existence. In the subsequent Section \ref{sec:place}
we discuss how to find such a place.
This seems to, unavoidably, require a search.

Most of the standard algorithms then proceed to compute a complete factorization of
the image of the polynomial when specialized to the residue field of the place just discussed (see the
description of a generic factorization algorithm in \cite{MR2168610}). 
This is known as the lifting and recombination strategy.
The factorization step in the residue field is often easy in practice but can be difficult in certain circumstances (see the beginning of Section \ref{sec:prelim}).
These algorithms also often have a bottleneck reconstructing global factorizations from local ones.
In contrast, our algorithm does not need to compute this factorization in the residue field at the beginning, nor tries to reconstruct
global factorizations from local ones. 
Instead, it may do a partial factorization during intermediate steps using easy gcd computations.
At the end, our algorithm may need to further factor some of these partial factors.
If the objective is an irreducibility criterion, it does not require finding such a factorization at all. 

There are two algorithms in the literature that use differential operators to factor polynomials and don't follow the lifting and recombination strategy.

The first is the algorithm of \cites{Rup, Gao2} that uses a certain first order partial differential
equation and shares some of the advantages of our approach. It only applies however when $K$ is
a rational function field. One application of this approach \cite[Theorem pg. 63]{Rup} is 
to bound the size of the largest prime $p$ for which an irreducible 
poynomial in $\mathbb{Z}[x,y]$ factors modulo $p$.
Our approach allows us to obtain similar bounds in full generality 
when $k$ itself is a global field.

The second is the algorithm of \cite{MR1937466} which uses the 
linear differential operator annihilating all roots of $G(T)$ 
(and also uses Wronskians). This algorithm is restricted to characteristic
zero as the crucial \cite[Proposition 4.2]{MR1937466} does not extend to positive
characteristic. Indeed, the proposition there is about the $k$-dimension of
the set of solutions in $K$ of the linear differential operator annihilating
roots of $G(T)$ (assumed irreducible). In characteristic $p$, the set of
solutions of a linear differential operator of order $<p$ in a field
$E/K$ is a $E^p$-vector
space. But even considering $E^p$-dimension, the set of solutions in
$K$ has the same $K^p$-dimension as the $E^p$-dimension of the
set of solutions in the splitting field $E$ of $G(T)$, as follows
from \cite[Lemma 1]{honda}, as opposed to having dimension $1$, which
would be the analogue of the characteristic zero result.

\section{Preliminaries}
\label{sec:prelim}


Let $K/k$ be a function field in one variable, that is, a finitely generated extension of transcendence degree one with
$k$ algebraically closed in $K$. Factoring polynomials in $K[T]$ includes
factoring polynomials in $k[T]$ and the latter could be hard, depending on
$k$. For example, if $k$ is a finite field or a number field, the
existence of a polynomial time deterministic factoring algorithm for
$k[T]$ is an open problem, although there are polynomial time probabilistic
algorithms that also perform well in practice. If $k$ is the field of real
numbers, there may be issues with precision. Our goal, therefore, is to
present an algorithm to factor polynomials in $K[T]$ relative to 
factoring polynomials in $k[T]$. 

Also, in discussing the running time of the algorithm, the estimates will be in terms of the number of 
field operations in $k$ or finite extensions thereof. 



\subsection{The basic rings}
\label{subsec:rings}

Let $v$ be a place of $K$ with ring of integers $\calO$ and maximal ideal
$\fm$. If $q$ is a power of the characteristic $p$ of $K$, we consider the ring $\calO/\fm^q$.
We have that the completion of $\calO$ is isomorphic to $\ell[[x]]$, where $\ell = \calO/\fm$ and
$x$ is an uniformiser of $v$, that is, an element of $\fm \setminus \fm^2$.
It follows that $\calO/\fm^q$ is isomorphic to $\ell[[x]]/(x^q)$.

Let $G(T) \in \calO[T]$ be a monic polynomial such that $G \pmod v$ is separable.

The ring $R_1 = (\calO/\fm^q) [T]/(G(T))$ is an Artinian ring and, thus,
a direct sum of Artinian local rings. These summands correspond to 
irreducible factors of $G(T)$ in $(\calO/\fm^q) [T]$ which, in turn, correspond to 
irreducible factors of $G(T) \in (\calO/\fm)[T]$. 
The standard factorization
version of Hensel's lemma gives an algorithmic way to go from the latter
to the former. 

Throughout, with notation as above and for a suitable value of $q$ computed in Step 3 of Algorithm \ref{alg:main}, we write $R = K[T]/(G(T))$, $R_0 = (\calO/\fm)[T]/(G(T))$ and $R_1 = (\calO/\fm^q)[T]/(G(T))$.

\subsection{Gaussian elimination}
\label{subsec:gauss}

In this subsection, we describe a Gaussian elimination procedure on an Artinian ring $A$ which is a sum of 
finitely many Artinian local rings (such as $R_1$ above). In addition to the usual row reduction steps (described explicitly below), the
algorithm will also involve splitting the ring in a direct sum of two rings and branching the algorithm to each 
summand. To ease notation, we continue to denote by $A$ any such summand. The final output of the procedure will be a decomposition of the original ring into a direct sum of rings and, for each such direct summand,
a matrix with entries in the corresponding subring. Moreover, for each maximal ideal of this subring, the image of the matrix in the quotient field will be row reduced and will have the pivots in the same place. In particular, the rank of the reduction will be independent of the maximal ideal of this subring and will be simply called the rank of the matrix.

The procedure is as follows, we scan each column in turn, looking for an element that is a unit in some summand of $A$. If none is found, we skip the column. If one is found, we split $A$ as a sum of two rings, one of which is the maximal summand where the element is a unit. In that summand we use the unit as a pivot (and in the other summand we move on to the next entry). Namely, we replace the other rows by the appropriate multiple of the row of the pivot so that the entry in the column of the pivot is $0$. Finally, we multiply the row of the pivot by its inverse, so the pivot is replaced by $1$. The usual analysis of Gaussian elimination justifies the claimed output. Indeed, we are just doing Gaussian elimination in the various quotient fields simultaneously to the extent possible and decomposing the ring as a direct sum, when it's not possible.

Finally, we note that the direct sum decomposition steps in the above procedure can be done explicitly as follows in the case 
of $R_1$. Namely, if the entry in the matrix being inspected is $P(T)$, we compute $H_0(T) = \gcd(G(T),P(T))$ and $E_0(T)=G(T)/H_0(T)$ in $(\calO/\fm) [T]$. If $E_0(T) \ne 1$, we lift the factorization $G(T)= H_0(T)E_0(T)$ in $R_0$ to a factorization $G(T)= H(T)E(T)$ in $R_1$
and decompose $R_1$ as the direct sum of $(\calO/\fm^q) [T]/(H(T))$ and $(\calO/\fm^q) [T]/(E(T))$. This works because we assume that $G(T)$ is squarefree in $(\calO/\fm) [T]$.

\subsection{Hasse derivatives and Wronskians}
\label{subsec:hasse}

We begin by presenting some concepts and results from \cites{Schmidt,SV}.
See also \cite{GV} which proves stronger versions of the main results of
\cite{Schmidt} and may be more accessible, as well as \cite[Section 6.1]{MR2041097} 
which describes
relevant algorithms (particularly algorithms 26 and 28 there). Let $K/k$ be a function field. The usual higher derivatives
do not work well in small characteristics. A suitable replacement for higher derivatives that work in general are the Hasse derivatives.
We denote by $D^{(i)}, i=0,1,\ldots,$ the Hasse derivatives with respect to some separating variable $x$ on $K$. These are $k$-linear operators on $K$ with
$i!D^{(i)} = (d/dx)^i$ and satisfying:

\begin{equation}
\label{eq:hasse}
    \begin{aligned}
D^{(i)} \circ D^{(j)} &= {{i+j}\choose{j}} D^{(i+j)},\\
D^{(i)}(uv) &= \sum_{j=0}^i D^{(j)}(u) D^{(i-j)}(v).
 \end{aligned}
\end{equation} 

The second formula is an extension of the Leibniz rule. They are defined first in $k[x]$ by setting $D^{(i)}x^n = {n\choose i}x^{n-i}$
and extending $k$-linearly. These operators then extend uniquely to $k(x)$ and any separable extension thereof by requiring that they satisfy the formulas \ref{eq:hasse} above.

By \cite[Satz 2]{Schmidt}, $f_0,\ldots,f_m \in K$ are linearly independent over $k$ if and only if 
the Wronskian matrix $(D^{(i)}(f_j))$ has maximal rank $m+1$. In this case, there is a (lexicographically) minimal 
list of integers $0=\varepsilon_0  < \cdots < \varepsilon_m$ such that the 
matrix $(D^{(\varepsilon_i)}(f_j))$ has maximal rank $m+1$.
We have $\varepsilon_i = i$ if the characteristic is zero or large enough 
but that is not going to be the situation in this paper.
 Also, to a place $v$ of $K$, we can 
associate a (lexicographically) minimal list of integers $0=j_0< \cdots < j_m$ (called the Hermitian invariants and which depend on $v$) 
such that the 
matrix $(D^{(j_i)}(f_j)\pmod v)$ has maximal rank $m+1$. Moreover, if $x$ is a local parameter at $v$, there exists a linear transformation of the space spanned 
by the $f_i$ over the residue field of $v$ that transforms $f_i$ into a basis (called an Hermitian basis) $g_i$ with $g_i = x^{j_i}+$ higher order terms (\cite[pg. 68]{Schmidt}). 

If $K$ is the function field of an algebraic curve $Y$, then the morphism $(f_0:\ldots:f_m): Y \to \mathbb{P}^m$
has some degree $\Delta$ and $\varepsilon_i \le j_i \le \Delta$ when the $f_i$
are linearly independent (\cite[Section 2]{SV},\cite[Prop. 13]{MR2041097}).
On the other hand, if the Wronskian matrix has rank $m$ and $a_0,\ldots,a_m \in K$
satisfy $\sum_{j=0}^m a_j D^{(i)}(f_j) = 0, i=0,1,2,\ldots$ and $a_0=1$,
then $a_j \in k, j=0,1,\ldots,m$ as follows from the proof of \cite[Satz 1]{Schmidt}.

Consider a monic, separable polynomial
\begin{equation}
\label{eq:poly}
G(T) = \sum_{j=0}^s a_jT^j, a_j \in K, a_s=1.
\end{equation} 

Let $R = K[T]/G(T)$ and $t$ the image of
$T$ in $R$. We extend the operators $D^{(i)}, i \ge 0$ to $R$.
We need an expression for $D^{(i)}(t)$. We have that
\begin{equation}
\label{eq:dg}
0=D^{(i)}(G(t)) = \sum_{j_1+2j_2+\cdots+ij_i \le i} 
A_{j_1,\ldots,j_i}(D^{(1)}(t))^{j_1}\cdots(D^{(i)}(t))^{j_i}
\end{equation}
where the $A_{j_1,\ldots,j_i}$ are polynomials in $t$ and, in particular,
$A_{0,\ldots,0,1} = G'(t)$ which is invertible in $R$ and
this determines $D^{(i)}(t)$ uniquely, by induction.
They can be computed more efficiently by the algorithms of \cite{MR2041097}.

If $v$ is a place of $K$ with ring of integers $\calO$ and maximal ideal
$\fm = (x)$ then, for any $q$ power
of $p$, the operators $D^{(i)}, i < q$,
defined using $x$ as the separating variable, 
induce operators on $\calO/\fm^q$.
Indeed, $D^{(i)}(x^q)=0, i < q$, so $D^{(i)}(x^qy)= x^qD^{(i)}(y), i < q$
for any $y \in \calO$ and it follows that $D^{(i)}, i < q$ preserve $\fm^q$.
In our main algorithm, we will choose $q$ large enough so that the non-zero
elements of $K$ appearing in the course of the computation have degree smaller than $q$,
so, if the computation produces an element in $\fm^q$, it has to be the zero element.

Now, if $G(T) \in \calO[T]$ is as above and is such that $G \pmod v$ is 
separable, then the operators $D^{(i)}, i < q$
induce operators on $R_1 = (\calO/\fm^q) [T]/(G(T))$ by the same formulas.
As mentioned before, $R_1$ is an Artinian ring and, thus,
a direct sum of Artinian local rings. 
For each irreducible factor $H_0$ of $G \pmod \fm$, there is a corresponding
factor $H$ of $G \pmod{\fm^q}$ and a corresponding summand 
$(\calO/\fm^q) [T]/(H(T))$ of $R_1$. If $\fm = (x)$, then
$\calO/\fm^q$ is isomorphic to $(\calO/\fm)[X]/(X^q)$ and
$(\calO/\fm^q) [T]/(H(T))$ is isomorphic to 
$((\calO/\fm)[T]/(H_0(T)))[X]/(X^q)$ and we refer to the subring 
$(\calO/\fm)[T]/(H_0(T))$ of this latter ring as its constants.

\begin{Lemma}
\label{lem:der}
If $u \in R_1$ satisfies $D^{(i)}(u)=0, 0 < i < q$, then in each local
summand of $R_1$, $u$ is constant (in the above sense).
\end{Lemma}

\begin{proof}
Each summand is isomorphic to $\ell[X]/(X^q)$ for some field $\ell$
and $D^{(i)}(x^r) = {r \choose i} x^{r-i}$, where $x$ is the image of $X$,
so it is clear that $D^{(i)}(u)=0, 0 < i < q$ if and only if $u \in \ell$.
\end{proof}

\begin{Remark}
In the process of algorithm \ref{alg:main} below, we will
decompose $R_1$ as a direct sum but we may not go all the way to the
full decomposition as a sum of Artinian local rings.
\end{Remark}

\section{The main algorithm}
\label{sec:algo}

Our main algorithm is described below as Algorithm \ref{alg:main}. As mentioned in the introduction, it takes as input a polynomial $G(T) \in K[T]$ and
a collection of finite dimensional $k$-vector spaces $V_i \subset K, i=0,\ldots r-1$ and decides whether there exists a factor $H(T)$ of $G(T)$ of the form $H(T) = \sum_{i=0}^r b_iT^i, b_i \in V_i, i < r, b_r=1$. Given $k$-bases $\{h_{ij}\}$ for the $V_i$, the algorithm proceeds by 
computing the Wronskian matrix of the functions $h_{ij}T^i$ in a suitable ring where $G(T)=0$ and doing Gaussian elimination on this matrix, 
in the sense of Subsection \ref{subsec:gauss}, identifying a $k$-linear dependence among the $h_{ij}T^i$.

To obtain a full factorization algorithm (where the $V_i$ are not necessarily given in advance), we prove lemma \ref{lem:poles} below 
(a variant of \cite[Lemma 4.1]{MR2168610}).

\begin{algorithm}
  \caption{Find factor of $G(T)$ with restricted coefficients
    \label{alg:main}}
\hspace*{\algorithmicindent} \textbf{Input} 
\begin{itemize}[label={}]
\item A function field $K/k$
\item A polynomial $G(T) \in K[T]$ monic, separable, of degree $s$ with discriminant $f \ne 0$.
\item A place $v$ of $K/k$, with ring of integers $\calO$ and maximal ideal $\fm$ with 
$G(T) \in \calO[T]$ and $v(f)=0$ and an uniformizer of $v$ used to define the $D^{(i)}$.
\item Finite dimensional $k$-vector spaces $V_i \subset K, i=0,\ldots r-1$, with $1 \in V_0$,
together with a $k$-basis $\{h_{ij}\}$ for each $V_i$, where $r < s$. Put $h_{01}=h_{r1}=1, m = \sum \dim V_i$.
\item An integer $\Delta$ which bounds $j_i, \varepsilon_i$ for any map to $\PP^m$ obtained by viewing $\Phi$ in a quotient field of $R$ in which its entries are linearly independent over $k$.
\end{itemize}
 \hspace*{\algorithmicindent} \textbf{Output} \\
 \begin{itemize}[label={}]
 \item Either a proof that $G(T)$ has no monic factor of the form
$H(T) = \sum_{i=0}^r b_iT^i, b_i \in V_i, i < r, b_r=1$.
\item Or a direct summand $S$ of the ring $R_1= (\calO/\fm^q) [T]/(G(T))$ and elements 
$b_i, i=0,\ldots,r$ of $S$ such that, for each local summand of $S$, a factor of $G(T)$ of
the required form can be obtained (see Theorem \ref{thm:main} and Remark \ref{rem:constant}).
\end{itemize}
  \begin{algorithmic}[1] 
    \State $R = K[T]/(G(T))$ and $t$ the image of $T$ in $R$. 
    \State $\Phi \in R^{m+1}$ the row vector with entries (in some order) $h_{ij}t^i \in R, i=0,\ldots,r$ and for each $i<r, j=1,\ldots,\dim V_i$ and $j=1$ for $i=r$.
    \State Compute $q$, smallest power of $p$ with $q > \max\{m,\Delta\}$ (so $q \le p \max\{m,\Delta\}$).
    \State $R_1 = (\calO/\fm^q) [T]/(G(T)), R_0=(\calO/\fm) [T]/(G(T))$.
    \State Compute matrix $M$ with rows $D^{(i)}(\Phi), i=0, \ldots, q-1$, working in $R_1$. 
    \State Do the Gaussian elimination on $M$, working in $R_1$, as described in Subsection \ref{subsec:gauss}.
     \For{Each direct summand $R'$ of $R_1$ returned by the Gaussian elimination step and corresponding matrix $M$}
        \If{$M$ has full rank $m+1$} 
           \State \Return{$G(T)$ has no factor of the form required for the output in $R'$.}
        \ElsIf{$M \pmod{\fm}$ has rank $m$}
           \State Compute solution $u_{ij}$ in $R'$ of $\sum_{ij} u_{ij} D^{(\ell)}(h_{ij}t^i) = 0,  u_{r1}=1, \ell=0,1,\ldots,q-1$.
           \State \Return{$b_i = \sum_{j} u_{ij} h_{ij}, i=0,\ldots,r$ in $R'$.}
         \Else
           \State{Go back to 2, replace the ring $R_1$ with the subring $R'$, remove an entry from $\Phi$ such that the corresponding column of the matrix $M$ has no pivot and set $m$ equal to $m-1$.}  
      \EndIf
      \EndFor
  \end{algorithmic}
\end{algorithm}


\begin{Lemma}
\label{lem:poles}
Let $G(T)=\sum_{i=0}^s a_i T^i \in K[T]$ with $a_s=1, a_0 \ne 0$ 
and let $\tau$ be a root of $G(T)$ in some finite extension $L/K$
and $H(T) = \sum_{i=0}^r b_i T^i$ be its monic minimal polynomial over $K$.
Then, for any place $v$ of $L$ extending a place of $K$,
$$v(\tau) \ge \min_{0,1,\ldots,s-1} v(a_i)/(s-i)$$
and, for any place $v$ of $K$,
$$v(b_j) \ge (r-j)\min_{0,1,\ldots,s-1} v(a_i)/(s-i).$$

\end{Lemma}

\begin{proof}

Recall that we assume throughout that $G(0)\ne 0$, so $\tau \ne 0$.
If $iv(\tau)+v(a_i) > sv(\tau)$ for all $i<s$, then
$\infty = v(G(\tau)) = \min\{iv(\tau)+v(a_i)\} = sv(\tau)$, contradiction. 
This gives the first part of the lemma.

We have that $b_j$ is the $(r-j)$-th elementary symmetric function on the conjugates
of $\tau$ so the second part follows from the first by extending $v$ to
a valuation of the splitting field of $H$.
\end{proof} 

Lemma \ref{lem:poles} provides bounds for the valuations of
the coefficients of the potential factors of $G(T)$ and these bounds can be used to define spaces $V_i$ such that if $G(T)$ factors, it has factors of the form $H(T) = \sum_{i=0}^r b_iT^i, b_i \in V_i, i < r, b_r=1$ for these $V_i$. To obtain $k$-bases for these $V_i$'s (as required in the algorithm) one may apply the results of \cite{HessRR}. It also follows that we can take 
\begin{equation}
\label{eq:Delta}
\Delta = -r\sum_v \min\{0, \min_{0,1,\ldots,s-1} v(a_i)/(s-i)\}
\end{equation}
as a bound for the $j_i, \varepsilon_i$ as required by Algorithm \ref{alg:main}.
Indeed this bound is required so that an element that appears in the computation as zero in $\calO/\fm^q$ is zero also in $K$. Any factor of $G(T)$ has coefficients of degree at most $q$, so if it appears as a solution over $\calO/\fm^q$ during our computation then we are assured that it is also a solution in $K$.
We also note that $\Delta = 0$ implies that all of the coefficients $a_i$ belong to $k$, and this case is discussed immediately following Remark \ref{rem:constant}.

We begin by proving that Algorithm \ref{alg:main} performs as described and runs in polynomial time. We defer a detailed estimate of the running time and a discussion of how to compute the auxiliary place to Section \ref{sec:place}.
\begin{Theorem}
\label{thm:main}
Given the above input, Algorithm \ref{alg:main} runs in deterministic polynomial time
in $p,s,\Delta$ (measured in number of operations in the field $\calO/\fm$) and 
outputs either a certificate that $G(T)$ has no factor of the required form or a decomposition
of $R_1$ as a direct sum of at most $s$ rings.
Moreover, for each such summand ($R'$, say), the algorithm outputs elements $u_{ij}$ of $R'$ 
that are constant in each summand of the decomposition of $R'$ into local rings and
from which a factor of $G(T)$ of the required form can be constructed or a certificate that this
summand does not yield such a factor. 
In particular, the algorithm provides a deterministic polynomial time absolute
irreducibility test in characteristic $p$ for $p$ polynomially bounded in $s, \Delta$ and a general factoring algorithm modulo
factoring in $k[T]$ under the same conditions.
\end{Theorem}

\begin{proof}

By induction on $m$. Assume $m=1$. Since $1 \in V_0, \Phi = (1,t^r)$ and $M$ has rank $1$ or $2$.
If the rank is $2$ in $R'$, it is clear there is no factor of the required form. If the rank is $1$ in $R'$, this means that $D^{(i)}(t^r)=0, i =1, \ldots, q-1$, so $t^r$ 
is constant in any local 
summand of $R'$, the solution of the linear system is $u_{01}=-t,u_{r1}=1$ and $H(T)=T-t$ is a factor of $G(T)$ of the required form. 

As mentioned above, step 3 is dealt with in general by equation \ref{eq:Delta} unless there is a better 
bound available. 


As mentioned above, the operators $D^{(i)}, i <q$ act on $R_1$ and the computation of the matrix $M$
in step 6 is polynomial in operations in $R_1$ but $\dim_{\calO/\fm} R_1 \le sq$ giving a bound in terms
of the number of operations in $\calO/\fm$.

The process of Gaussian elimination has running time polynomial in the size
of the matrix just as in the field case.At the end of it, we arrive at a decomposition of $R$ as a sum of at most $s$
rings and the image of $M$ in each of these is put in row echelon form.

Those summands of $R_1$ where $M$ has full rank $m+1$ yield no factor of
$G(T)$ of the required form, since such a factor is a linear relation among the entries of $\Phi$ with constant coefficients and, applying $D^{(i)}, i <q$ to this relation shows that the coefficients are in the kernel of $M$. For other factors $R'$ where $M$ has rank $m$, we compute the
$u_{ij}$ as described in step 11. It follows from the proof of \cite[Theorem 1]{GV}
that $D^{(r)}(u_{ij})=0, 0 < r < q$ and thus, from Lemma \ref{lem:der} the $u_{ij}$ are constant in each local factor of the decomposition of $R'$. 
Let $S$ be one such factor. It determines a factor $H(T)$ of $G(T)$ and a place $w|v$ in the function field $L$ obtained 
by adjoining a root $t$ of $H(T)$ to $K$. Moreover, if ${\hat u}_{ij}$ are the constants in $S$ corresponding to the
$u_{ij}$, the fact that $\sum u_{ij}h_{ij}t^i=0$ in $S$, shows that $\sum {\hat u}_{ij}h_{ij}t^i$ vanishes to order
at least $q$ at $w$. But $q > \Delta$ and $\Delta$ is an upper bound for the degree of the morphism associated
to $\Phi$, therefore an upper bound for the degree of $\sum {\hat u}_{ij}h_{ij}t^i$ and this implies that $\sum {\hat u}_{ij}h_{ij}t^i=0$ in $L$. Hence $\sum {\hat u}_{ij}h_{ij}T^i$ yields a factor of $G(T)$ of the required form, if the ${\hat u}_{ij}$ are elements of $k$ and, otherwise, $S$ does not yield a factor of $G(T)$ of the required form.
For the factors $R'$ where $M$ has rank smaller than $m$, we decrement $m$, drop an entry of $\Phi$ as described, rerun the algorithm and are done by induction.

For the last two claims of the theorem, first notice that when $G(T)$ is absolutely irreducible, the matrices $M$ will have full rank $m+1$
in every factor ring in play on the algorithm so the algorithm will terminate with a certificate of absolute irreducibility. Finally, a general factoring algorithm follows from Lemma \ref{lem:poles} and the discussion surrounding it.
\end{proof}

\begin{Remark}
\label{rem:constant}
What Algorithm \ref{alg:main} and Theorem \ref{thm:main} don't do is to identify the $u_{ij}$ with
specific elements of an algebraic extension of $k$, necessarily. For that, we need to further factor
the factor of $G(T) \pmod \fm$ corresponding to the ring $R'$ as a product of irreducibles
to obtain the full decomposition of $R'$ as a sum of local rings
and identify the $u_{ij}$ with constants in each summand. Note that it is possible that $G(T) \pmod \fm$
has irreducible factors of degree greater than $1$ and, in the local ring corresponding to that factor,
$\sum {\hat u}_{ij}h_{ij}T^i$ will not give a factor of $G(T)$ of the required form over $k$. If such an
irreducible factor of $G(T) \pmod \fm$ is detected, the corresponding local ring can be discarded.
For each linear factor of $G(T) \pmod \fm$, once it is identified, we obtain a summand of $R'$ which is
a local ring with residue field $k$ and projecting each ${\hat u}_{ij}$ to the residue fields yields the $u_{ij}$,
completing the factorization process.

As mentioned in the introduction, the first step in most standard factoring algorithms 
for $K[T]$ is to fully factor $G(T) \in (\calO/\fm)[T]$
and how it is performed depends on the
nature of the field $k$. In our algorithm, this step
may not be required at all (if $G(T)$ has no factors, or a single irreducible 
factor, of the required form) or it may only be needed for a proper factor of $G(T)$. 

On the other hand, it is not claimed that the algorithm, as described, will output all factors of the required form, only that it will output at least one such factor, in case it exists. This can be easily remedied by dividing by the output factors and rerunning the algorithm with the quotient as input.
\end{Remark}

We will discuss an example but, beforehand, here is a non-example. If $G(T) \in k[T]$,
that is, has constant coefficients, then $D^{(i)}(t)=0, i >0$, the matrix $M$ has always 
rank one and the algorithm unravels to the base case $m=1$. The polynomial $T-t$ is
a factor of $G(T)$ for $G(t)=0$ and we are left with the task of factoring $G(T)$ over
the constant field $k$.

For a more representative example consider 
$$G(T)= T^4 + (x+1)T^3 + (x^2+1)T^2+(x^3+x^2+1)T+(x^2+x) \in \F_2(x)[T].$$
We look for factors of $G(T)$ of the form $T+ax+b$, so $r=1$, $V_0$ is spanned
by $1,x$, $\Phi = (1,x,t)$ where $t$ is the image of $T$ in $(k[x]/(x^4))[T]/(G(T))$ and $m=2$. Modulo the ideal $(x+1)$, $G(T)$ reduces
to $T^4+T$. We find, using \ref{eq:dg} that $D^{(2)}(t) = J(t)/G'(t)^3$, where
$$J(T) = (xT^5 + (x^2 + x)T^4 + x^5T + (x^6 + x^5)).$$

The gcd of $G(T)$ and $J(T)$ is $H(T)=T^2+T + x^2+x$ and, switching to
the ring $R_1=(k[x]/(x+1)^4)[T]/(H(T))$, we let $t$ be the image of $T$, forcing $D^{(2)}(t) = 0$. 
Using \ref{eq:dg} now with $H$, we find that $D(t)=1$.
The matrix $M$ is therefore
$$M=
\begin{pmatrix}
1&x&t\\
0&1&1\\
0&0&0
\end{pmatrix}
$$
Now, we solve the system $t+ax+b= D(t)+a=0$, so
$a=D(t)=1, b=t+x$.
Now, modulo the ideal $(x+1)$, we have $H(T) = T^2+T = T(T+1)$ and we lift this
factorization to $k[x]/(x+1)^4$ and find $H(T) = (T+x)(T+x+1)$. Now, $R_1$ is a direct sum
of two rings from this factorization and $b=0,1$ respectively in each of the factors.
Consequently, $G(T)$ has the factors $T+x, T+x+1$ of the required form.

\begin{Remark}
If the characteristic is zero, $R_1$ is not defined and if the characteristic is large, $R_1$ is too big. An extension of the algorithm of this paper to those cases would require a replacement for $R_1$. For theoretical results, such as Theorem \ref{thm:height}, we can work with $R=\calO [T]/(G(T))$ but for algorithmic purposes we need a ring finite dimensional over $k$. We 
could work with the ring $(\calO/\fm^n) [T]/(G(T))$ for some appropriate choice of $n$, but
the higher derivatives are not typically well-defined as operators on this ring. Instead, it makes sense to look at
$D^{(i)}: (\calO/\fm^n) [T]/(G(T)) \to (\calO/\fm^{n-i}) [T]/(G(T))$.
We have not worked out the full details of this possibility.
\end{Remark}

\section{Running time estimates and finding a suitable place}
\label{sec:place}

We retain the notation of the previous section and especially of Algorithm
\ref{alg:main}.

\begin{Theorem}
\label{thm:running}
Given a place $v$ as in the input of Algorithm \ref{alg:main}
with $[\calO/\fm:k] =d$, the algorithm runs in 
$O(ds(p\max\{m,\Delta\})^4)$ field operations in $k$.
\end{Theorem}

\begin{proof}
The ring $R_1$ is initialized in the algorithm and during its execution
it is replaced by its direct summands, of which there are at most $s$.
From \cite[Section 6.1]{MR2041097}, the computation of the matrix
$M$ takes $O(q^2)$ operations in each ring $R_1$ and Gaussian elimination
of the matrix $M$ takes $O(q^3)$ operations in each ring $R_1$. The latter
dominates the former. The initial ring $R_1$ is a vector space over
$\calO/\fm$ of dimension $sq$, so we can bound the running time of these
calculations as $O(sq^4)$ operations in $\calO/\fm$. Using now that
$q \le p\max\{m,\Delta\}$ we arrive at a running time of
$O(s(p\max\{m,\Delta\})^4)$ operations in $\calO/\fm$. If we let
$d= [\calO/\fm:k]$, we finally arrive at $O(ds(p\max\{m,\Delta\})^4)$
operations in $k$. 
\end{proof}

\begin{Remark}
There are faster algorithms in the special cases of most interest,
e.g. $k$ finite, $K=k(x)$ \cites{MR2600710,MR3709331} (e.g. \cite[Theorem 1]{MR2600710}, seems to
translate to a running time of $O(ds(\max\{m,\Delta\})^3)$ field operations in $k$) or finding a
root of $G(T)$, again for $k$ finite \cite{Gao}. There does not seem to
be an algorithm with the same generality and flexibility with a stated
running time for direct comparison.
\end{Remark}

It remains to make explicit how to find a suitable
place $v$ and compute $d$.  We are given a function field $K/k$ and $f \in K, f \ne 0$, we need to construct a place $v$ of $K$ such that $v(f)=0$, as well as an uniformizer for this place. Without
loss of generality, we assume that $k$ is finitely generated over its prime field. We can also assume that
$f \notin k$, for otherwise the condition $v(f)=0$ is automatic and we can replace $f$ by an element of $K \setminus k$
in order to produce the place $v$. 
Under this additional hypothesis, there are only at most $2[K:k(f)]$ places of $K$
with $v(f) \ne 0$ as $[K:k(f)]=\sum_v \max\{0,v(f)\} = -\sum_v \min\{0,v(f)\}$.
Hence, to find $v$ it suffices to generate enough places of $K$
until a place with $v(f)=0$ is found.
It seems that such a search is, at the moment, unavoidable. In the case $k$ finite, $K=k(x)$, for example, $f(x) = a(x)/b(x), a(x),b(x) \in k[x]$ 
and all that is required is an element $\alpha$ (possibly in an extension of $k$) with $a(\alpha)b(\alpha) \ne 0$ which can be found
by a simpleminded search over a suitably large extension of $k$.
However, such a search can be more 
difficult in the generality we work with and some of the literature seems to gloss over this point.
We assume $K/k(f)$ separable. For a general discussion of how to reduce to
this case, see \cite{Steel}.
Then $K=k(f,g)$ for some $g \in K$ and there
is $P(x,y) \in k[x,y], P(f,g)=0$. We then find 
$\alpha \ne 0$ in $k$ or in an extension field such that $P(\alpha,y)$ is separable.
There is a place of $K$ corresponding to each irreducible factor of
$P(\alpha,y)$ over $k(\alpha)$ and, if $m(x)$ is the minimal polynomial 
of $\alpha$ over $k$, then $m(f)$ is an uniformizer for any such place,
as $m(f)$ is an uniformizer for the corresponding place of $k(f)$ which,
by construction, is unramified in $K$. 
An easy way of ensuring that $P(\alpha,y)$ is separable is to require
that it be irreducible in $k(\alpha)[y]$. This has the advantage that
irreducibility tests are easier than factorization and that $m(x)=x-\alpha$
in this case. On the other hand, this makes the degree of the place bigger.

Assume first that $k$ is a finite field. This case is discussed in
\cite[Algorithm 3.2]{Gao} with some additional assumptions which are 
relevant there but not here, so we just give a simplified discussion.
Basically, one just searches for $\alpha$ in $k$ or an extension thereof
until one is found with $P(\alpha,y)$ is irreducible in $k(\alpha)[y]$. Chebotarev's 
density theorem implies that $\alpha$ exists if $\# k$ is large enough in terms of
$[K:k(f)]$ and the genus of $K$.
It is worth pointing out that we may not have a suitable place
with residue field $k$ or even an extension of $k$ of small degree as
we can see by considering a simple example such as 
$G(T)=T^p -(x^{p^n}-x)T+1 \in \F_p(x)[T]$ (which has $f=\pm (x^{p^n}-x)^p$). 

The case where $k$ is not algebraic over its prime field can be tackled
as follows. Let $k_0$ be the algebraic closure in $k$ of its prime field.
Then $K$ and $k$ are respectively the function field of varieties $X,Y$
over $k_0$ with a map $X \to Y$ of relative dimension $1$. We realize
$X$ as a hypersurface of projective space and intersect $X$ with random
hypersurfaces. As long as these hypersurfaces intersect $X$ in an irreducible
subset that is not a component of the divisor of $f$ and is transversal to
the generic fiber of the map $X \to Y$, such a hypersurface will define a
place of $K/k$ satisfying our requirements and the equation of the 
hypersurface is the uniformizer we need. The version of Bertini's 
theorem over finite fields from \cite{MR3402695} guarantees that, for high 
enough degree, most hypersurfaces satisfy our conditions. Checking irreducibility can be done using the main algorithm of this paper (applied to
a field of smaller transcendence degree) and the other conditions can be
checked directly.
Alternatively, one can use the explicit version of Hilbert's irreducibility
theorem for function fields from \cite{MR4252300} to estimate how far one might need
to search.

\section{Other applications}
\label{sec:appl}

\subsection{Arithmetic properties}

As mentioned in the introduction, \cite{Rup} has a bound on the size
of the largest prime $p$ for which an irreducible
polynomial in $\mathbb{Z}[x,y]$ factors modulo $p$.
We substantially extend this result. We will consider polynomials in
$K[T]$ where $K/k$ is a function field and $k$ itself is a global field.
We impose no restriction on the characteristic of $k$.
We can associate a height to elements of $k$ in the usual way (\cite[Chapter 3, \S 3]{Lang}),
namely, for $a \in k$, the height of $a$ is
$-\sum_v \min\{0,v(a)\}$, where $v$ runs through the places of $k$. If we represent $K$ as
a finite extension of $k(x)$ for some transcendental element $x$ of $K$,
for an element $a \in K$, there is a polynomial $P(x,y) \in k[x,y], P(x,a)=0$
of minimal degree. We call the coefficients (in $k$) of $P$ simply the
coefficients of $a$ and use the maximum of their heights as a measure of the
complexity of $a$.
Given $G(T)\in K[T]$, for all but finitely primes $\fp$ of $k$, 
we can consider the reduction of $G(T)$ modulo $\fp$.

\begin{Theorem}
\label{thm:height}
Let $k$ be a global field, $K/k$ a function field and $G(T) \in K[T]$
an irreducible polynomial, as in equation \ref{eq:poly} and define $\Delta$
as in equation \ref{eq:Delta}. Let $H$ be the maximum height of the coefficients
of the $a_i$. The norm $N(\fp)$ of the primes $\fp$ of $k$ for which the
reduction of $G(T)$ modulo $\fp$ is either undefined or reducible satisfies
$N(\fp) = O(H^{(\Delta+1)^3})$, where the implied constant depends on $s, \Delta$.

\end{Theorem}

\begin{proof}
Since $G(T)$ is assumed irreducible, the matrix $M$ in Algorithm \ref{alg:main}
has maximal rank and so a maximal minor has non-zero determinant. Note that
$\Delta + 1$ is a bound for the number of columns (hence also of rows) of $M$.
We now compute $M$ with entries in $R$ (with no restriction on the characteristic 
of $k$). It follows by induction from equation \ref{eq:dg} that the height of the coefficients
of $D^{(i)}(t^j)$ is $O(H^{ij}) = O(H^{(\Delta+1)^2})$, hence the height of the coefficients
of the (non-zero) determinant of a maximal $(\Delta+1)\times (\Delta+1)$-minor of $M$ is 
$O(H^{(\Delta+1)^3})$. But for the
reduction of $G(T)$ modulo $\fp$ to be either undefined or reducible, $\fp$ needs to
divide the determinant of this minor and the result follows.
\end{proof}

In the case where $k=\F_q(Z)$, the approach of the
previous theorem can be used to prove a proximity gap 
statement (in the sense of \cites{prox,BN}) for Algebraic Geometry codes.

\subsection{List decoding}
\label{subsec:list}

A different application concerns the 
the Guruswami--Sudan list-decoding algorithm for Reed--Solomon
or Algebraic Geometry codes \cite{Gao}.
If one has a fixed code, there are fixed finite dimensional $k$-vector 
spaces $V_0, W_0,\ldots,W_{s-1} \subset K$ and, for each received message, a 
polynomial $G(T) \in K[T]$ is constructed with coefficients $a_i \in W_i$
(the precise construction is
irrelevant at the moment), for which we want to know
its roots in $V_0$. That is precisely the problem we dealt with above,
with $r=1$. 

For an alternative approach to improving the factorization step of
the Guruswami--Sudan algorithm, see \cite{MR1806819}. Additionally, \cite[Theorem 5]{MR1806819} provides an alternative way of finding a suitable place (in the sense of Section \ref{sec:place}) to be used in either approach to factorization in the context of the Guruswami--Sudan algorithm.


For a prime power $q$ and an integer $k,  1 \le k \le q$, the Reed-Solomon code $RS_k(q)$ is the subspace of the $q$-dimensional space of functions $\F_q \to \F_q$ corresponding to the functions represented by a polynomial $f(x) \in \F_q[x], \deg f < k$. Define, as usual for $g_1,g_2: \F_q \to \F_q$, the Hamming distance 
$d(g_1,g_2)= \# \{\al \in \F_q \mid g_1(\al) \ne g_2(\al)\}$.

McEliece \cite{McE} has shown that, for a fixed Reed-Solomon code, most invocations of the Guruswami--Sudan algorithm output a list of codewords of size $0$ or $1$. We will strengthen this result and show that, for a certain range of Reed-Solomon codes, for most invocations of the Guruswami--Sudan algorithm, the polynomial $G(T) \in K[T]$ has at most
one root in $V_0$, so we can take full advantage of our factoring algorithm's enhanced perfomance in this situation. Our result is meant to be illustrative and one might expect a similar result in much greater generality.

The following proposition sets the stage for our result and is a special case of Sudan's original precursor \cite{Sudan} to the Guruswami--Sudan algorithm.

\begin{Proposition}
\label{prop:sudan}
Assume $(q+5)/10 < k \le q/8$ and that $g: \F_q \to \F_q$ is any function. Then there exists a non-zero polynomial 
$G(x,y) = a_0(x) +a_1(x)y+a_2(x)y^2+a_3(x)y^3+a_4(x)y^4 \in \F_q[x,y]$ with $\deg a_i \le (4-i)(k-1), i=0,\ldots, 4$
such that $\forall \al \in \F_q, G(\al,g(\al))=0$. Moreover, if $f(x) \in RS_k(q)$ satisfies $d(f,g) \le [q/2]$ then $G(x,f(x))$
is identically zero.
\end{Proposition}

\begin{proof}
Each condition $G(\al,g(\al))=0$ is a linear equation on the coefficients of the $a_i$. The total number of coefficients
is $\sum_{i=0}^4 ((4-i)(k-1)+1) = 10k -5 > q$, by hypothesis. So there is a $G$, as prescribed. If $f(x) \in RS_k(q)$ satisfies $d(f,g) \le [q/2]$ then $G(x,f(x))$ is a polynomial with at least $q/2$ zeros but, since $\deg f \le  k - 1$, we have 
$\deg G(x,f(x)) \le 4(k-1) < q/2$ by hypothesis, hence $G(x,f(x))$ is identically zero.
\end{proof}

So, as mentioned above, given such $G$, we need to find its factors of the form $y-f(x), \deg f < k$ to find the elements 
$f(x) \in RS_k(q), d(f,g) \le [q/2]$ and, by construction, there are at most $4$ such. This is the ``list'' in the list-decoding algorithm. It is easy to see that, for most choices of $g$, the list is empty. Indeed, there are $q^q$ possible $g$'s and at most 
$q^k\sum_{i=0}^{[q/2]} \binom{q}{i} (q-1)^i$ of those satisfy $d(f,g) \le [q/2]$ for some $f \in RS_k(q)$ and the sum is asymptotic to
$q^{k+q/2+o(q)}$, hence bounded above by $q^{5q/8 + o(q)}$, much smaller than $q^q$.

More interestingly, McEliece \cite{McE} has shown that, for most choices of $g$ such that there exists some $f \in RS_k(q)$ with $d(f,g) \le [q/2]$, the $f$ is unique. It follows that $q^{k+q/2+o(q)}$ is the correct order of magnitude for the number of $g$ with $d(f,g) \le [q/2]$ for some $f \in RS_k(q)$. The following result strengthens this result by showing that, for most choices of $g$ such that there exists some $f \in RS_k(q)$ with $d(f,g) \le [q/2]$, any polynomial $G$, output of Proposition \ref{prop:sudan}, has only one factor of the form $y-f(x), \deg f < k$. This justifies our claim that a factoring algorithm that performs better when there is only one factor of prescribed type improves the running time of the list decoding algorithm.

\begin{Theorem}
Assume $(q+5)/10 < k < q/8$. Among those $g: \F_q \to \F_q$ such that there exists some $f \in RS_k(q)$ with $d(f,g) \le q/2$, the number of those for which a polynomial $G$, output of Proposition \ref{prop:sudan}, has more than one factor of the form $y-f(x), \deg f < k$
is at most $q^{5k + o(q)}$.
\end{Theorem}

\begin{proof}
We assume that $d(f_0,g) = \delta \le [q/2]$, for some $f_0$ as above. Replacing $g$ by $g-f_0$, where $f_0$ is one choice, we can assume $f_0=0$. By Proposition \ref{prop:sudan}, $G(x,y)=yH(x,y)$, for some $H$. Moreover, $H(\al, g(\al))=0$ for $\delta$ values of $\al$. We split the proof in two cases, depending on whether or not there exists $\alpha \in \F_q$ with $H(\alpha,y)$ identically zero.

Suppose first that there is no such $\al$. The number of $g$'s for a given choice of $H$ is at most $\binom{q}{\delta} 3^{\delta}$. Indeed, there are $\binom{q}{\delta}$ choices for a subset of size $\delta$ of $\F_q$ and, for each $\al$ in this subset, at most $3$ choices for a solution of $H(\al,y)=0$ as $H$ is cubic in $y$ (for the $\al$ outside of the subset we set $g(\al)=0$). If, in turn, $H(x,y)=(y-f(x))S(x,y)$ with $f \in RS_k(q)$ and 
$$S(x,y)=b_0(x) +b_1(x)y+b_2(x)y^2, \deg b_i \le (2-i)(k-1), i=0,1, 2,$$
 we can count the number of coefficients, namely $k$ for $f$ and $3k$ for the $b_i$'s. Hence at most $q^{4k}$ possible such $H$'s and at most $\sum_{\delta\le [q/2]}\binom{q}{\delta} 3^{\delta}q^{4k}$ possible such $g$'s. If we now allow the initial $f_0$ to vary (as opposed to be set to $f_0=0$) we end up with a total of at most 
$\sum_{\delta\le [q/2]} \binom{q}{\delta} 3^{\delta} q^{5k}$ for the count of the theorem and this is at most $q^{5k + o(q)}$, as desired.

Let's now assume that there are $\alpha \in \F_q$ with $H(\alpha,y)$ identically zero and let $c(x)$ be the monic polynomial that has these $\alpha$'s as roots.
Since the coefficient of $y^3$ in $H(x,y)$ as a polynomial in $y$ is constant, we conclude in this case that it is zero. We can then write $H$ (in the case that $G$ has a second root $f \in RS_k(q)$) as
$$H(x,y) = c(x)(y-f(x))(b_0(x)/c(x) +(b_1(x)/c(x))y).$$ 
If we fix $\deg c = m$,
the number of possible $H$'s is at most $q^m q^k q^{2k-m} q^{k-m} = q^{4k-m}$ by counting the number of possible $c,f,b_0/c,b_1/c$ respectively. The number of $g$ given $H$ is bounded above by $3^{q-m} q^m$ by counting the options for $g(\al)$ when $c(\al) \ne 0$ or $c(\al) = 0$, respectively. So, the count in this case is at most $q^{4k + o(q)}$ and again, allowing $m$ and the initial $f_0$ to vary (as opposed to be set to $f_0=0$) we end up with a total of at most $q^{5k + o(q)}$ $g$'s, completing the proof.

\end{proof}

We note that $5k < k + q/2$ if $k<q/8$ which justifies the claim made immediately before the theorem. We also note that there are functions, as in the statement of the theorem, for which $G$ has more than one factor of the required form. For example, if $q$ is even, we can take any $g$ that takes only values $0,1$ each $q/2$ times. Then $G=y(y+1)$ is an output of Proposition \ref{prop:sudan} for this $g$.

Finally, we note that Algorithm \ref{alg:main} vastly simplifies in the context of Proposition \ref{prop:sudan}. In this case $K = \F_q(x), T=y, r=1$. The vector 
$\Phi$ is $(1,x,x^2,\ldots,x^{k-1},y)$, $m=k, \Delta = 4k$ and the matrix $M$ is already in echelon form with rows $(0,0,0,\ldots,0,D^{(\varepsilon)}(y))$ after the first $k$ rows. The only real computation is to express the pivots $D^{(\varepsilon)}(y)$ as elements of $R_1$ and do the computations producing the splitting of $R_1$ in subrings from Subsection \ref{subsec:gauss}. As we want to detect whether $y$ is a polynomial of degree $<k$, it is enough to check whether $D^{(\varepsilon)}(y) = 0$ for $k \le \varepsilon < q$ and 
depending on the base-$p$ expansion of $k$, only a few values of $\varepsilon$ need to be checked, e.g. if $k$ is a power of $p$, then only the powers of $p$ satisfying $k \le \varepsilon < q$ need to be checked.



\section*{Acknowledgements}
This research was funded by the Ministry for Business, Innvovation and Employment in New Zealand.
I would also like to thank 
M. Esgin, V. Kuchta, S. Ruj, A. Sakzad and R. Steinfeld of the Trans-Tasman ZK group
for questions that motivated part of this research
and a number of referees for their very thorough and useful reports.

\section*{Statements}

The author has no competing interests to declare that are relevant to the content of this article.

No datasets were generated or analyzed during the current study.

	


\begin{bibdiv}
\begin{biblist}

\bib{MR1806819}{article}{
   author={Augot, Daniel},
   author={Pecquet, Lancelot},
   title={A Hensel lifting to replace factorization in list-decoding of
   algebraic-geometric and Reed-Solomon codes},
   journal={IEEE Trans. Inform. Theory},
   volume={46},
   date={2000},
   number={7},
   pages={2605--2614},
}

\bib{MR4252300}{article}{
   author={Bary-Soroker, Lior},
   author={Entin, Alexei},
   title={Explicit Hilbert's irreducibility theorem in function fields},
   conference={
      title={Abelian varieties and number theory},
   },
   book={
      series={Contemp. Math.},
      volume={767},
      publisher={Amer. Math. Soc., Providence, RI},
   },
   date={2021},
   pages={125--134},
}
		

\bib{MR2537701}{article}{
   author={Belabas, Karim},
   author={van Hoeij, Mark},
   author={Kl\"{u}ners, J\"{u}rgen},
   author={Steel, Allan},
   title={Factoring polynomials over global fields},
   journal={J. Th\'{e}or. Nombres Bordeaux},
   volume={21},
   date={2009},
   number={1},
   pages={15--39},
   issn={1246-7405},
}

\bib{prox}{article}{
    title = {Proximity Gaps for Reed-Solomon Codes},
   author = {Ben-Sasson, Eli },
   author={Carmon, Dan},
   author={Ishai, Yuval},
   author={Kopparty, Swastik},
   author={Saraf, Shubhangi},
    note = {Cryptology ePrint Archive, Report 2020/654},
    year = {2020},
}

\bib{MR3128698}{article}{
   author={Berthomieu, J\'{e}r\'{e}my},
   author={Lecerf, Gr\'{e}goire},
   author={Quintin, Guillaume},
   title={Polynomial root finding over local rings and application to error
   correcting codes},
   journal={Appl. Algebra Engrg. Comm. Comput.},
   volume={24},
   date={2013},
   number={6},
   pages={413--443},
}

\bib{BN}{article}{
      title={Interactive Oracle Proofs of Proximity to Algebraic Geometry Codes}, 
      author={Bordage, Sarah},
      author={Lhotel, Mathieu},
      author={Nardi, Jade},
      author={Randriam, Hugues},
      booktitle =	{37th Computational Complexity Conference (CCC 2022)},
      pages =	{30:1--30:45},
      series =	{Leibniz International Proceedings in Informatics (LIPIcs)},
      year =	{2022},
      volume =	{234},
      publisher =	{Schloss Dagstuhl -- Leibniz-Zentrum f{\"u}r Informatik},
  address =	{Dagstuhl, Germany},
}

\bib{MR3402695}{article}{
   author={Charles, Fran\c{c}ois},
   author={Poonen, Bjorn},
   title={Bertini irreducibility theorems over finite fields},
   journal={J. Amer. Math. Soc.},
   volume={29},
   date={2016},
   number={1},
   pages={81--94},
}


\bib{MR1937466}{article}{
   author={Cormier, Olivier},
   author={Singer, Michael F.},
   author={Trager, Barry M.},
   author={Ulmer, Felix},
   title={Linear differential operators for polynomial equations},
   journal={J. Symbolic Comput.},
   volume={34},
   date={2002},
   number={5},
   pages={355--398},
}


\bib{Gao}{article}{
   author={Gao, Shuhong},
   author={Shokrollahi, M. Amin},
   title={Computing roots of polynomials over function fields of curves},
   conference={
      title={Coding theory and cryptography},
      address={Annapolis, MD},
      date={1998},
   },
   book={
      publisher={Springer, Berlin},
   },
   date={2000},
   pages={214--228},
}

\bib{Gao2}{article}{
   author={Gao, Shuhong},
   title={Factoring multivariate polynomials via partial differential
   equations},
   journal={Math. Comp.},
   volume={72},
   date={2003},
   number={242},
   pages={801--822},
}


\bib{GV}{article}{
   author={Garc\'{\i}a, Arnaldo},
   author={Voloch, Jos\'{e} Felipe},
   title={Wronskians and linear independence in fields of prime
   characteristic},
   journal={Manuscripta Math.},
   volume={59},
   date={1987},
   number={4},
   pages={457--469},
   }

\bib{MR790658}{article}{
   author={von zur Gathen, J.},
   author={Kaltofen, E.},
   title={Factorization of multivariate polynomials over finite fields},
   journal={Math. Comp.},
   volume={45},
   date={1985},
   number={171},
   pages={251--261},
}

\bib{MR2041097}{article}{
   author={Hess, Florian},
   title={An algorithm for computing Weierstrass points},
   conference={
      title={Algorithmic number theory},
      address={Sydney},
      date={2002},
   },
   book={
      series={Lecture Notes in Comput. Sci.},
      volume={2369},
      publisher={Springer, Berlin},
   },
   date={2002},
   pages={357--371},
}

\bib{HessRR}{article}{
   author={Hess, F.},
   title={Computing Riemann-Roch spaces in algebraic function fields and
   related topics},
   journal={J. Symbolic Comput.},
   volume={33},
   date={2002},
   number={4},
   pages={425--445},
}

\bib{honda}{article}{
   author={Honda, Taira},
   title={Algebraic differential equations},
   conference={
      title={Symposia Mathematica, Vol. XXIV},
      address={Sympos., INDAM, Rome},
      date={1979},
   },
   book={
      publisher={Academic Press, London-New York},
   },
   date={1981},
   pages={169--204},
}


\bib{Lang}{book}{
   author={Lang, Serge},
   title={Fundamentals of Diophantine geometry},
   publisher={Springer-Verlag, New York},
   date={1983},
   pages={xviii+370},
   isbn={0-387-90837-4},
   review={\MR{715605}},
   doi={10.1007/978-1-4757-1810-2},
}

\bib{MR2600710}{article}{
   author={Lecerf, Gr\'{e}goire},
   title={New recombination algorithms for bivariate polynomial factorization based on Hensel lifting},
   journal={Appl. Algebra Engrg. Comm. Comput.},
   volume={21},
   date={2010},
   number={2},
   pages={151--176},
   }
 
\bib{McE}{article}{
author={R. J. McEliece}
title={On the Average List Size for the Guruswami-Sudan Decoder}
conference={ title={7th International Symposium on Communications Theory and Applications (ISCTA)},}
date={2003}
}



\bib{MR3703706}{article}{
  author={Neiger, Vincent},
   author={Rosenkilde, Johan},
   author={Schost, \'{E}ric},
   title={Fast computation of the roots of polynomials over the ring of power series},
   conference={ title={ISSAC'17---Proceedings of the 2017 ACM International Symposium on Symbolic and Algebraic Computation}, },
   book={ publisher={ACM, New York}, },
   date={2017},
   pages={349--356},
}

		
\bib{MR2168610}{article}{
   author={Pohst, Michael E.},
   title={Factoring polynomials over global fields. I},
   journal={J. Symbolic Comput.},
   volume={39},
   date={2005},
   number={6},
   pages={617--630},
}

\bib{Rup}{article}{
   author={Ruppert, Wolfgang M.},
   title={Reducibility of polynomials $f(x,y)$ modulo $p$},
   journal={J. Number Theory},
   volume={77},
   date={1999},
   number={1},
   pages={62--70},
}
	
\bib{Schmidt}{article}{
   author = {Schmidt, Friedrich Karl},
	journal = {Mathematische Zeitschrift},
	number = {1},
	pages = {62--74},
	title = {Die Wronskische Determinante in beliebigen differenzierbaren Funktionenk{\"o}rpern},
	volume = {45},
	year = {1939},
	}   
   
\bib{Steel}{article}{
   author={Steel, Allan},
   title={Conquering inseparability: primary decomposition and multivariate
   factorization over algebraic function fields of positive characteristic},
  journal={J. Symbolic Comput.},
  date={2005},
   number={3},
   pages={1053--1075},
}


\bib{SV}{article}{
   author={St\"{o}hr, Karl-Otto},
   author={Voloch, Jos\'{e} Felipe},
   title={Weierstrass points and curves over finite fields},
   journal={Proc. London Math. Soc. (3)},
   volume={52},
   date={1986},
   number={1},
   pages={1--19},
   }
   
   \bib{Sudan}{article}{
   author={Sudan, Madhu},
   title={Decoding of Reed Solomon codes beyond the error-correction bound},
   journal={J. Complexity},
   volume={13},
   date={1997},
   number={1},
   pages={180--193},
  }

\bib{MR3709331}{article}{
   author={Weimann, Martin},
   title={Bivariate factorization using a critical fiber},
   journal={Found. Comput. Math.},
   volume={17},
   date={2017},
   number={5},
   pages={1219--1263},
   }

 
\end{biblist}
\end{bibdiv}
\end{document}